\documentclass[12pt,a4paper]{article}

\usepackage{amssymb, amsmath, amsthm}

\usepackage[cm]{fullpage}
\usepackage[english]{babel}
\usepackage[cp1250]{inputenc}
\usepackage[T1]{fontenc}
\usepackage[pdftex]{graphicx}
\usepackage{booktabs}

\newtheorem{thm}{Theorem}
\newtheorem{lem}{Lemma}
\newtheorem{prop}{Proposition}

\newtheorem{cor}{Corollary} 

\title{Finite difference method for a Volterra equation with a power-type nonlinearity}
\author{Hanna Okrasi{\'n}ska-P{\l}ociniczak\thanks{Department of Mathematics, Wroclaw University of Environmental and Life Sciences, ul. C.K. Norwida 25, 50-275 Wroclaw, Poland}$\;$, \L ukasz P\l ociniczak\thanks{Faculty of Pure and Applied Mathematics, Wroc{\l}aw University of Science and Technology, Wyb. Wyspia{\'n}skiego 27, 50-370 Wroc{\l}aw, Poland}$\;^,$\footnote{Email: lukasz.plociniczak@pwr.edu.pl, \underline{Corresponding Author}}}
\date{}

\begin{document}
\maketitle

\begin{abstract}
In this work we prove that a family of explicit numerical finite-difference methods is convergent when applied to a nonlinear Volterra equation with a power-type nonlinearity. In that case the kernel is not of Lipschitz type, therefore the classical analysis cannot be applied. We indicate several difficulties that arise in the proofs and show how they can be remedied. The tools that we use consist of variations on discreet Gronwall's lemmas and comparison theorems. Additionally, we give an upper bound on the convergence order. We conclude the paper with a construction of a convergent method and apply it for solving some examples. \\

\end{abstract}

\section{Introduction}
Integral equations are one of most useful tools used in mathematical analysis and modelling. Many essential problems formulated as differential equations can be transformed into the equivalent integral equations which, many times, are more useful for showing existence, uniqueness and deriving estimates. This field has reached its maturity and there is a wealth of literature reviewing many subjects such as the theory for Volterra \cite{Bru17,Gri90,Cor91} and Fredholm equations \cite{Tri57,Sta00,Kre89}, numerical methods \cite{Atk76,Lin85,Bru17,Hac12,Bru04,Bru84} and topics in integro-differential equations \cite{Caf09,O12} and fractional calculus \cite{Mil93,Gor97,Mai10,Sam93}.

In this work we consider a family of numerical finite-difference methods for solving a certain class of nonlinear Volterra equations. These, in particular, arise as models of dynamics in porous media \cite{Vaz07,Okr89,Okr95}, heat transfer \cite{Man51}, propagation of shock-waves in gas filled tubes \cite{Kel81} and anomalous diffusion \cite{Plo14,Plo13}. The latter application is the main motivation behind our investigations (more details can be found in \cite{Plo15}). The general form of the family of equations that we consider can be written as follows
\begin{equation}
	u(x) = \int_0^x K(x,t) u(t)^\frac{1}{m+1} dt, \quad m>0, \quad x\in[0,X].
\label{eqn:IntEq0}
\end{equation}
Notice that the nonlinearity is not Lipschitz so that we cannot use the most of the classical theory. However, there are many results concerning existence and uniqueness. In the seminal papers \cite{Bus76,Gri81} the problem was stated and several conditions for existence and uniqueness were given. Later, in a number of work those results were strengthened and generalized (see \cite{Okr89,Bus86,Bus89,Myd91,Nie97}). Abel integral equations with power-type nonlinearities have also been recently investigated \cite{Buc05,Wan14}. For further details see other works of cited authors. 

According to the best of authors' knowledge, the literature consists of only very few papers concerning numerical methods for this kind of equations. In \cite{Buc97} a iterative fixed-point way of solving (\ref{eqn:IntEq0}) was proposed. Also, in \cite{Cor00} a short review article concerning theoretical and numerical concepts of nonlinear Volterra equation has been published. 

In what follows we consider a family of explicit methods for solving (\ref{eqn:IntEq0}) and prove that under certain boundedness assumptions on the kernel they are convergent. Moreover, we find a bound on the convergence rate and illustrate the theory with the numerical simulations. 	

\section{Finite difference method}
Consider the following nonlinear Volterra equation which arises from (\ref{eqn:IntEq0}) via the transformation $y(x)^{m+1} = u(x)$
\begin{equation}
	y(x)^{m+1} = \int_0^x K(x,t) y(t) dt, \quad m > 0, \quad x\in[0,X].
\label{eqn:IntEq}
\end{equation}
We will need the boundedness assumption on the kernel $K\in C([0,X]^2)$
\begin{equation}
	0<C\leq K(x,t) \leq D,
\label{eqn:KBound}
\end{equation}
from which only the lower one is a nontrivial requirement. It can be shown that the positive solution of (\ref{eqn:IntEq}) with (\ref{eqn:KBound}) exists and is unique (see \cite{Bus76} and other papers mentioned in the Introduction). 

In order to construct the numerical method fix a natural number $N$ and introduce the grid 
\begin{equation}
	x_n = n h, \quad h = \frac{X}{N}, \quad n = 0,1,2,...,N.
\end{equation}
The next step is to discretize the integral in (\ref{eqn:IntEq}), say
\begin{equation}
	\int_0^{x_n} K(x_n,t) y(t) dt = h \sum_{i=1}^{n-1} w_{n,i} K(x_n,x_i) y(x_i) + \delta_n(h),
\label{eqn:LocCons}
\end{equation}
where $\delta_n(h)$ is the local consistency error. Next, we define 
\begin{equation}
	\delta(h) := \max_{1\leq n\leq N} |\delta_n(h)|,
\label{eqn:delta}
\end{equation}
and further assume that
\begin{equation}
	0<w_{n,i} \leq W, \quad n,i = 0,1,2,...,N,
\label{eqn:W}
\end{equation}
for some $W>0$. Denoting $y_n$ as a numerical approximation to $y(nh)$ we may then propose the following \emph{explicit} finite difference scheme for solving (\ref{eqn:IntEq})
\begin{equation}
	y_n^{m+1} = h \sum_{i=1}^{n-1} w_{n,i} K_{n,i} \; y_i, \quad n=2,3,...,N,
\label{eqn:NumMet}
\end{equation} 
where $K_{n,i} := K(x_n,x_i)$. Since the equation (\ref{eqn:IntEq}) has a trivial solution, it is necessary to start the above iteration with a value which will force the convergence to the nontrivial one. In the next section we will show one way of making that choice.

Below we will prove that (\ref{eqn:NumMet}) is convergent to the unique positive solution of (\ref{eqn:IntEq}). Before that, however, we need some auxiliary results. The first is a simple observation concerning the iteration scheme (\ref{eqn:NumMet}).
\begin{prop}
Let $y$ be a solution of (\ref{eqn:IntEq}) and construct $y_n$ via the iteration (\ref{eqn:NumMet}). If $\delta_n(h)$ in (\ref{eqn:LocCons}) is positive (negative) for $n=1,2,...,N$, then $y(n h) \geq y_n$ ($y(nh) \leq y_n$) provided that $y(h) \geq y_1$ ($y(h)\leq y_1$).
\label{prop:Mono}
\end{prop}
\begin{proof}
It suffices to consider only the case $\delta_n\geq 0$ for all $n=1,2,...,N$. Since by assumption we have $y(h) \geq y_1$ then, by induction, we can assume that the assertion holds for the $n$ first terms. The inductive step can be conducted as follows
\begin{equation}
\begin{split}
	y((n+1)h)^{m+1} &= \int_0^{(n+1)h} K((n+1)h,t) y(t) dt = h\sum_{i=1}^{n} w_{n+1,i} K_{n+1,i} \; y(ih) + \delta_{n+1}(h) \\
	&\geq h\sum_{i=1}^{n} w_{n+1,i} K_{n+1,i} \; y_i + \delta_{n+1}(h) \geq y_{n+1}^{m+1}.
\end{split}
\end{equation}
The first inequality is a consequence of the inductive assumption, while in the second we use the fact that $\delta_{n+1} \geq 0$ and (\ref{eqn:NumMet}). It follows that $y((n+1)h) \geq y_{n+1}$ what ends the proof.
\end{proof}

Next, we prove a comparison theorem that can be thought as a generalization of the Gronwall-Bellman's Lemma.
\begin{lem}
\label{lem:LemGron}
	Let $y\in C{[0,1]}$ be a positive function satisfying
	\begin{equation}
		y(x)^{m+1} \lessgtr C x^\mu \int_0^x t^\nu y(t) dt,
	\end{equation}
	with $\nu,\mu,m>0$. Then
	\begin{equation}
		y(x) \lessgtr \left(C\;\frac{m}{m+1}\frac{1}{1+\nu+\frac{\mu}{m+1}}\right)^\frac{1}{m} x^\frac{1+\nu+\mu}{m},
	\end{equation}
where the sign $\lessgtr$ denotes either less or greater than.
\end{lem}
\begin{proof}
The proof will proceed for the $\geq$-case. The reasoning for the other follows exactly the same route. Let $y$ be as in the assumption, and denote
\begin{equation}
	z(x) := \int_0^x t^\nu y(t) dt,
\end{equation}
with $z(0)=0$. We have then 
\begin{equation}
	y(x) \geq C^\frac{1}{m+1} x^\frac{\mu}{m+1} z(x)^\frac{1}{m+1},
\label{eqn:yEst}
\end{equation}
and therefore
\begin{equation}
	z'(x) = x^\nu y(x) \geq C^\frac{1}{m+1} x^{\nu+\frac{\mu}{m+1}} z(x)^\frac{1}{m+1}.
\end{equation}
After division by $z^\frac{1}{m+1}$ we can integrate to obtain
\begin{equation}
	\int_0^x z(t)^{-\frac{1}{m+1}}z'(t)dt \geq C^\frac{1}{m+1} \int_0^x t^{\nu+\frac{\mu}{m+1}} dt,
\end{equation}
whence
\begin{equation}
	\frac{m+1}{m} z(x)^{\frac{m}{m+1}} \geq C^\frac{1}{m+1} \frac{1}{1+\nu+\frac{\mu}{m+1}} x^{1+\nu+\frac{\mu}{m+1}}.
\end{equation}
Using the above estimate in (\ref{eqn:yEst}) implies
\begin{equation}
	y(x) \geq \left(C\;\frac{m}{m+1}\frac{1}{1+\nu+\frac{\mu}{m+1}}\right)^\frac{1}{m} x^\frac{1+\nu+\mu}{m}.
\end{equation}
This ends the proof.
\end{proof}

In a similar vein we can prove the following asymptotic result. 
\begin{lem}
	Let $y\in C{[0,1]}$ be a positive solution of (\ref{eqn:IntEq}) with
	\begin{equation}
		K(x,t) \sim C x^\mu \quad \text{as} \quad x\rightarrow 0^+ \; \text{and} \; t \leq x.
	\end{equation}
	Then
	\begin{equation}
	y(x) \sim \left(C\;\frac{m}{m+1}\frac{1}{1+\frac{\mu}{m+1}}\right)^\frac{1}{m} x^\frac{1+\mu}{m}, \quad x\rightarrow 0^+.
	\end{equation}
\label{lem:Asym}
\end{lem}
\begin{proof}
We start by deriving an representation of the solution to (\ref{eqn:IntEq}). By the Mean Value Theorem we have
\begin{equation}
	y(x)^{m+1} = K(x,\tau(x)) \int_0^x y(t) dt,
\end{equation}
because $y$ is positive while $0\leq \tau(x) \leq x$. Similarly as before we substitute
\begin{equation}
	z(x) := \int_0^x y(t) dt,
\end{equation}
and therefore 
\begin{equation}
	z'(x) = y(x) = \left(K(x,\tau(x)) z(x)\right)^\frac{1}{m+1}.
\end{equation}
Solving this differential equation yields
\begin{equation}
	z(x) = \left(\frac{m}{m+1}\int_0^x K(t,\tau(t))^\frac{1}{m+1}dt\right)^\frac{m+1}{m}.
\end{equation}
Finally,
\begin{equation}
	y(x) = K(x,\tau(x))^\frac{1}{m+1}\left(\frac{m}{m+1}\int_0^x K(t,\tau(t))^\frac{1}{m+1}dt\right)^\frac{1}{m}.
\end{equation}
If now $K(x,t)\sim C x^\mu$ as $x\rightarrow 0^+$ with $t\leq x$ we have
\begin{equation}
	y(x) \sim \left(C x^\mu\right)^\frac{1}{m+1}\left(C^\frac{1}{m+1}\frac{m}{m+1}\frac{1}{1+\frac{\mu}{m+1}}x^{1+\frac{\mu}{m+1}}\right)^\frac{1}{m}.
\end{equation}
The proof is complete.
\end{proof}

The second lemma is a discrete analogue of the above.
\begin{lem}
Let $\left\{y_n\right\}$, $n=1,2,...,N$ be a sequence of positive numbers satisfying the recurrence (\ref{eqn:NumMet}) with (\ref{eqn:KBound}) and (\ref{eqn:LocCons}). Assume that the initial value is chosen to satisfy
\begin{equation}
	y_1 \geq \left(1-\frac{m+1}{m} \frac{\epsilon(h)}{h^{1+\frac{1}{m}}}\right)\left(C \frac{m}{m+1} h \right)^\frac{1}{m},
\label{eqn:LemIterationInit}
\end{equation}
then we have
\begin{equation}
	y_n \geq \left(1-\frac{m+1}{m} \frac{\epsilon(h)}{h^{1+\frac{1}{m}}}\right)\left(C \frac{m}{m+1} nh \right)^\frac{1}{m} , \quad n\geq 2,
\label{eqn:LemIteration}
\end{equation}
where $\epsilon(h)$ is the maximal local consistency error (\ref{eqn:delta}) of the quadrature (\ref{eqn:LocCons}) applied to the function $t^\frac{1}{m}$.
\label{lem:LemIteration}
\end{lem}
\begin{proof}
Due to the assumption we can proceed to the inductive step. To this end, let $y_i$ satisfy (\ref{eqn:LemIteration}) for every $i=1,2,...n$. Observe that from (\ref{eqn:NumMet}) with (\ref{eqn:KBound}) we have
\begin{equation}
	y_{n+1}^{m+1} \geq C\left(1-\frac{m+1}{m} \frac{\epsilon(h)}{h^{1+\frac{1}{m}}}\right)\left(C \frac{m}{m+1}\right)^\frac{1}{m} h \sum_{i=1}^{n} w_{n,i} (i h)^\frac{1}{m}.
\end{equation}
Because the sum above is a discretization of the integral (\ref{eqn:LocCons}) we can further write
\begin{equation}
\begin{split}
y_{n+1}^{m+1} &\geq C\left(1-\frac{m+1}{m} \frac{\epsilon(h)}{h^{1+\frac{1}{m}}}\right)\left(C \frac{m}{m+1}\right)^\frac{1}{m} \left(\int_0^{(n+1)h} t^\frac{1}{m}dt - \epsilon_{n+1}(h)\right) \\
& =  \left(1-\frac{m+1}{m}\frac{\epsilon(h)}{h^{1+\frac{1}{m}}}\right)\left(C \frac{m}{m+1}(n+1)h\right)^\frac{m+1}{m}\left(1-\frac{m}{m+1}\frac{\epsilon_{n+1}(h)}{\left((n+1)h\right)^\frac{m+1}{m}}\right),
\end{split}
\end{equation}
where $\epsilon_{n+1}$ is the $n+1$th local consistency error for the quadrature. Now, in order to complete the induction we have to show that
\begin{equation}
	 1-\frac{m}{m+1}\frac{\epsilon_{n+1}(h)}{\left((n+1)h\right)^\frac{m+1}{m}}\geq \left(1-\frac{m+1}{m} \frac{\epsilon(h)}{h^{1+\frac{1}{m}}}\right)^m.
\end{equation}
Since by definition, $\epsilon_{n+1} \leq \epsilon$ and obviously $(n+1)h\geq h$ we further have
\begin{equation}
1-\frac{m}{m+1}\frac{\epsilon(h)}{h^\frac{m+1}{m}}\geq \left(1-\frac{m+1}{m} \frac{\epsilon(h)}{h^{1+\frac{1}{m}}}\right)^m.
\end{equation}
Now, due to convexity we have the inequality $(1-x)^m \leq 1-x$ for $m\geq 1$ which grants the validity of the above estimate. This ends the proof.
\end{proof}
Finally, we prove a variant of the discreet generalization of the Gronwall-Bellman's lemma. 
\begin{lem}
\label{lem:LemRecur}
	Let $\left\{e_n\right\}$, $n=1,2,...$ be a sequence of positive numbers satisfying 
	\begin{equation}
		e_n \leq \frac{1}{n} \left(A \sum_{i=1}^{n-1} e_i + B\right), \quad n\geq 2,
	\label{eqn:LemRecur}
	\end{equation}
	for $A > 0$. If we define
		\begin{equation}
		M := \max\{e_1,B\}\times\left\{
		\begin{array}{ll}
		\dfrac{1}{-A \zeta(1-A)}, & 0<A<1 \\
		1, & A\geq 1,
		\end{array}
		\right.
	\label{eqn:LemRecurAssum}
	\end{equation}
	then
	\begin{equation}
		e_n \leq \frac{M}{n^{1-A}}.
		\label{eqn:LemRecurSol}
	\end{equation}
\end{lem}
\begin{proof}
First, notice that $-1<A \zeta(1-A) < 0$ for $0<A<1$ \cite{Tit86}. Therefore, the above bounds are well-defined. To prove them we proceed by induction. Assume (\ref{eqn:LemRecurAssum}) and notice that obviously we have $e_1 \leq M$. 

%

Next, we make the inductive step and assume that (\ref{eqn:LemRecurSol}) holds for all terms up to a fixed $n$. We will show that this assertion is also true for the $(n+1)$th term. First, assume that $A\geq 1$ and use the inductive assumption
\begin{equation}
	e_{n+1} \leq \frac{M}{n+1} \left(A \sum_{i=1}^{n} i^{A-1}+\frac{B}{M}\right).
\end{equation}
Now, we need an estimate for the above sum. Since the function $x^{A-1}$ is increasing, we can bound the series by an integral
\begin{equation}
	\sum_{i=1}^{n} i^{A-1} \leq \int_1^{n+1} x^{A-1} dx = \frac{1}{A}\left((n+1)^A-1\right),
\end{equation}
therefore
\begin{equation}
	e_{n+1} \leq \frac{M}{n+1} \left((n+1)^A-1+\frac{B}{M}\right) \leq \frac{M}{(n+1)^{1-A}},
\end{equation}
since $B\leq M$. On the other hand, for the case $0<A<1$ we have
\begin{equation}
	e_{n+1} \leq \frac{M}{n+1} \left(A\sum_{i=1}^{n}\frac{1}{i^{1-A}} + \frac{B}{M}\right).
\end{equation}
The appropriate bound for the sum comes from the Riemann-Siegel formula \cite{Tit86}, which gives
\begin{equation}
	\sum_{i=1}^{n}\frac{1}{i^{1-A}} \leq \frac{n^A}{A}  + \zeta(1-A),
\end{equation}
whence
\begin{equation}
	e_{n+1} \leq \frac{M}{n+1} \left(n^A+A\zeta(1-A)+\frac{B}{M}\right) \leq \frac{M}{n+1} (n+1)^A = \frac{M}{(n+1)^{1-A}},
\end{equation}
where the last inequality follows from the fact that $B \leq - M A\zeta(1-A)$. This concludes the reasoning.
\end{proof}
As we noted, the above two lemmas are generalizations of the Gronwall-Bellman's results. Similar theorems arise in many situations when considering integral equations and numerical methods used in solving them. There is a large collection of papers about various generalizations of the aforementioned inequalities from which the monograph \cite{Ame97} gives a thorough exposition of the continuous case. Many other results concerning the discreet version can be found for instance in \cite{Ye07,Tao03,Dix86,Mc82}.

We can now proceed to the main theorem. We claim that the numerical method (\ref{eqn:NumMet}) is locally convergent to the unique positive solution of (\ref{eqn:IntEq}). We difficulty of the proof lies in the non-Lipschitz behaviour of the nonlinearity in (\ref{eqn:IntEq}). The usual method do not work since we cannot make estimate which will invoke the standard Gronwall-Bellman's lemma. To overcome this difficulty we use a comparison result which helps us to bound the solution from below with a known function for which the numerical method is convergent. 

\begin{thm}
	Let $y=y(x)$ be the nontrivial solution of (\ref{eqn:IntEq}) with (\ref{eqn:KBound}). Moreover, by $e_n := y(nh) - y_n$ denote the error of the numerical scheme (\ref{eqn:LocCons})-(\ref{eqn:NumMet}) and assume that $\epsilon(h)h^{-\frac{m+1}{m}}<1$ as $h\rightarrow 0^+$, where $\epsilon(h)$ is defined in Lemma \ref{lem:LemIteration}. Then, if $y_1$ satisfies (\ref{eqn:LemIterationInit}) with $\left|e_1\right| = O(h^p)$ for some $p>0$, we have
	\begin{equation}
		\left|e_n\right| \leq  \text{const.} \times h^{1-\frac{1}{m}\frac{WD}{C E(h)}}\max\{h^p,\delta(h) h^{-1}\} \quad \text{as} \quad h\rightarrow 0^+ \quad \text{with} \quad nh \rightarrow x_n, 
	\label{eqn:Rate}
	\end{equation}
	where $E(h) := 1-\frac{m+1}{m}\epsilon(h)h^{-\frac{m+1}{m}} =O(1)$ as $h\rightarrow 0^+$.
\label{thm:Conv}
\end{thm}
\begin{proof}
First, we can use the Langrange's mean-value theorem
\begin{equation}
	\left|y(nh)^{m+1}-y_n^{m+1}\right| = (m+1) \xi^m \left|y(nh)-y_n\right|, \quad m>0,
\label{eqn:xi}
\end{equation}
where $\xi$ is between $y(nh)$ and $y_n$. Now, we use the above estimate with (\ref{eqn:IntEq}) and (\ref{eqn:LocCons}) to obtain
\begin{equation}
	(m+1) \xi^m |e_n| \leq \left|y(nh)^{m+1}-y_{n}^{m+1}\right| \leq h W D \sum_{i=1}^{n-1} |e_i| + \delta.
\end{equation}

Now, we have to consider two cases depending on the relative size of $y(nh)$ and $y_n$. If $y_n \leq y(nh)$, then $\xi \geq y_n$ and from Lemma \ref{lem:LemIteration} it follows that
\begin{equation}
m C \; nh \left(1-\frac{m+1}{m}\frac{\epsilon(h)}{h^\frac{m+1}{m}} \right) |e_n|\leq (m+1)  |e_n| \leq h W D\sum_{i=1}^{n-1} |e_i| + \delta.
\end{equation}
Eventually, defining 
\begin{equation}
	E(h) := 1-\frac{m+1}{m}\frac{\epsilon(h)}{h^\frac{m+1}{m}} 
\end{equation}
we can write
\begin{equation}
|e_n| \leq \frac{1}{n} \left(\frac{1}{m}\frac{WD}{C E(h)}\sum_{i=1}^{n-1} |e_i| + \frac{1}{m C}\frac{\delta}{h E(h)}\right).
\end{equation}
On the other hand, if $y(nh)\leq y_n$, then $\xi \geq y(nh)$ and
\begin{equation}
m C \; nh\leq (m+1)  |e_n| \leq h W D\sum_{i=1}^{n-1} |e_i| + \delta.
\label{eqn:enEst}
\end{equation}
where in the first inequality we have used Lemma \ref{lem:LemGron}. Hence
\begin{equation}
|e_n| \leq \frac{1}{n} \left(\frac{1}{m}\frac{WD}{C}\sum_{i=1}^{n-1} |e_i| + \frac{1}{m C}\frac{\delta}{h}\right) \leq \frac{1}{n} \left(\frac{1}{m}\frac{WD}{C E(h)}\sum_{i=1}^{n-1} |e_i| + \frac{1}{m C}\frac{\delta}{h E(h)}\right).
\end{equation}
since $E(h) < 1$ due to the assumption. It follows that regardless of the relative size of $y(nh)$ and $y_n$, the error term satisfies the above recurrence inequality. Now, we can invoke the Lemma \ref{lem:LemRecur} to conclude that since $E(h) = O(1)$ as $h\rightarrow 0^+$ we have
\begin{equation}
	|e_n| \leq \text{const.} \times \frac{\max\{e_1,\delta h^{-1}\}}{n^{1-A}} = \text{const.} \times \frac{\max\{h^p,\delta h^{-1}\}}{h^{A-1}} \frac{1}{(nh)^{1-A}},
\end{equation}
where $A:=\frac{1}{m}\frac{WD}{C E(h)}$. Now, when $h\rightarrow 0^+$ with $nh \rightarrow \text{const.}$ we obtain the assertion. The proof is complete.
\end{proof}
The analysis of the above proof can lead to a stronger convergence result for a specific family of quadratures. 
\begin{cor}
Let $y=y(x)$ be a solution of (\ref{eqn:IntEq}) with (\ref{eqn:KBound}). Moreover, choose a quadrature with $\delta_n \leq 0$ for each $n=1,2,...,N$. If the starting value $y_1$ has $|e_1| = O(h^p)$ for some $p>0$, then 
\begin{equation}
	\left|e_n\right| \leq  \text{const.} \times h^{1-\frac{1}{m}\frac{WD}{C }}\max\{h^p,\delta(h) h^{-1}\} \quad \text{as} \quad h\rightarrow 0^+ \quad \text{with} \quad nh \rightarrow x_n.
\label{eqn:Rate2}
\end{equation}
\label{cor:Conv}
\end{cor}
\begin{proof}
If $\delta_n \leq 0$, then by Proposition \ref{prop:Mono} we have $y(nh) \leq y_n$. The proof follows the same steps as in Theorem \ref{thm:Conv} with the difference that we always have $\xi \geq y(nh)$ in (\ref{eqn:xi}). Hence the conclusion can be drawn from (\ref{eqn:enEst}) without any subsequent estimates. 
\end{proof}

\section{Construction of a quadrature and numerical examples}
In this section we will construct a convergent numerical method for solving (\ref{eqn:NumMet}). By Lemma \ref{lem:Asym}, we can identify several problems in numerically solving this integral equation. For example, when $m>1$ the non-smooth behaviour of the solution at $x=0$ can severely slow down the convergence. Recall that for a smooth function, its Newton-Cotes quadrature has an error bounded by the value of a sufficiently high derivative. If the considered function is not smooth enough, the order of the method can be reduced \cite{Plo17,Ded05}. On the other hand, when $0<m<1$ the Corollary \ref{cor:Conv} can give very weak estimates on the convergence rate. Having these remarks in mind, in what follows, we will illustrate the above results by constructing a numerical method for the case $m\geq 1$ leaving the other case for the future work.

To maximize the theoretical convergence rate we will use a quadrature with $\delta_{n}(h) \leq 0$ and use Corollary \ref{cor:Conv}. Notice that $y(x) \propto x^\frac{1}{m}$ so that for $m\geq 1$ the solution is locally concave. As a quadrature we choose the midpoint method which is of second order for sufficiently smooth functions. However, near $x=0$ the method will be only $O(h^{1+\frac{1}{m}})$. In order to see this fact we use the result from \cite{Ded05} stating that for any function with bounded variation the local consistency error is the following
\begin{equation}
	\left|\int_0^h f(t) dt - h f\left(\frac{h}{2}\right)\right| = \frac{h}{2} V_0^h (f).
\end{equation}
Now, due to the monotonicity we have $V_0^h (t^\frac{1}{m}) = h^\frac{1}{m}$, which states that the quadrature has order $1+\frac{1}{m}$. Note that this is in contrast with usual estimate for smooth functions.

Moreover, we can restrict ourselves only to an arbitrary and fixed interval $[0,\xi]$, so that the derivative singularity at the origin would be integrated. For the rest of the interval any other quadrature can be used since the nonlinearity of the equation (\ref{eqn:IntEq}) is of the Lipschitz class.

In order to convince oneself that the midpoint quadrature will yield a greater area under a curve than the exact integral consider an arbitrary interval $[x_i, x_{i+1}]$. Construct a line tangent to $x_{i+\frac{1}{2}}$ and notice that due to concavity, the area of a trapezium generated by it is greater than the are area under $y$. On the other hand, the trapezium theorem states that its area is equal to the area of a rectangle spanned by its height and the midline. This precisely is the midpoint quadrature. These considerations can be confirmed by the standard error analysis but this reasoning requires existence of a second derivative. 

Before we construct the quadrature we have to choose the starting value for the iteration (\ref{eqn:NumMet}). We have several ways of doing this. Keeping in mind that we have to choose $y_1 \geq y(h)$, we can try to bound the defining integral equation by some simple, not necessarily open, quadrature. The simplest choice is to use the rectangle rule 
\begin{equation}
	y(h)^{m+1} = \int_0^h K(h,t) y(t) dt \leq y(h) \int_0^h K(h,t) dt ,
\end{equation}
because $y$ is increasing and $K$ positive. We can thus define the initial step for the iteration as
\begin{equation}
	y_1 := \left(\int_0^h K(h,t) dt\right)^\frac{1}{m}.
\label{eqn:y1}
\end{equation}
This choice has a disadvantage of producing a large starting error $O(h^\frac{1}{m})$, since by Lemma \ref{lem:Asym} we have
\begin{equation}
	\left|y(h) - y_1 \right| \sim \left(1-\left(\frac{m}{m+1}\right)^\frac{1}{m}\right)(Ch)^\frac{1}{m} \quad \text{as} \quad h\rightarrow 0^+.
\end{equation}
Therefore we can put $p=\frac{1}{m}$ in the Corollary \ref{cor:Conv}. In order to increase the accuracy we can always use Richardson's extrapolation and we implicitly assume that it is utilized in our quadrature.

Having made all the preparations we can propose the following iteration scheme to solve (\ref{eqn:IntEq})
\begin{equation}
	y_{2n+k}^{m+1} = \frac{1}{2}h K_{2n+k,k} \;y_k + 2 h \sum_{i=1}^{n} K_{2n+k,2i+k-1} \; y_{2i+k-1}, \quad k\in\left\{0,1\right\}, \quad n > 1.
\label{eqn:Midpoint}
\end{equation}
The first term $y_1$ is calculated by the formula (\ref{eqn:y1}) and $y_0 = 0$. This quadrature uses intervals of length $2h$ and approximates the integral by the midpoint rule. Therefore, we have to distinguish between the even and odd steps ($k=0,1$ respectively) with regard to whether or not the terminal point is a midpoint of the interval. The first term on the right-hand side of the above formula is nonzero only in the odd term and is a result of a trapezoid rule applied to the interval $[0,h]$. We will illustrate the above results with some numerical examples. \\

\noindent\textbf{Example 1.} As a first example we consider the following equation
\begin{equation}
	y(x)^{m+1} = \int_0^x y(t)dt, \quad x\in[0,1],
\label{eqn:Ex1}
\end{equation}
which can be readily solved with
\begin{equation}
y(x) = \left(\frac{m}{m+1}\;x\right)^\frac{1}{m}.
\end{equation}
Here, $K(x,t) \equiv C=D=1$ and hence, thanks to (\ref{eqn:Midpoint}) we have $W =2$. From (\ref{eqn:Rate}) we immediately deduce that the scheme is convergent with order at least equal to $1+\frac{1}{m}-\frac{2}{m}=1=\frac{1}{m}$. Recall, that since the solution is proportional to $x^{\frac{1}{m}}$, the midpoint method will converge with an order of $1+\frac{1}{m}$.

The results for numerical simulations are given in Tab. \ref{tab:Ex1}. We can see that the real order of convergence depends very weakly on $m$ and is close to $1$. Note that we are comparing the exact with numerical solution at a fixed point $x=0.001$. In order to compute the order we use the standard technique of linear regression on a log-log plot of the error versus the iteration step which, in our simulation, changes from $10^{-1}$ to $10^{-5}$. Notice that the integration is exact for $m=1$ which is not surprising since midpoint rule is exact for the linear functions. \\

\begin{table}
	\centering
	\begin{tabular}{ccccccc}
		\toprule
		m	  &  $1$ & $1.5$ & $2$ & $10$ & $10^2$ & $10^3$ \\
		\midrule
		Order & $\infty$ & 1.000 & 0.999 & 0.98 & 0.927 & 0.918  \\
		Est. order & 0 & $0.333$ & $0.500$ & $0.900$ & $0.990$ & $0.999$ \\
		\bottomrule
	\end{tabular}
	\label{tab:Ex1}
	\caption{Simulated order of convergence of the midpoint rule (\ref{eqn:Midpoint}) applied to (\ref{eqn:Ex1}). Calculation has been done at $x=0.001$.}
\end{table}

\noindent\textbf{Example 2.} The next equation we consider is the following equation with a convolution kernel
\begin{equation}
y^{m+1}(x) = \int_0^x e^{x-t} y(t)dt, \quad x\in[0,X],
\label{eqn:Ex2}
\end{equation}
which is solved by
\begin{equation}
y(x) = e^\frac{x}{m+1}\left(1-e^{-\frac{m}{m+1}x}\right)^\frac{1}{m}.
\end{equation}
Here, $C=1$, $D=e^{X}$ and $W = 2$, hence the order of convergence can be estimated as equal to at least $1+\frac{1}{m}-\frac{2}{m} e^{X} > 1-\frac{1}{m} $. Choosing a sufficiently small $X$ assures that the method is convergent (specifically $X<\ln \frac{m+1}{2}$). Results of the simulations are given in Tab. \ref{tab:Ex2}. Once again, the order of convergence is very close to $1$.

\begin{table}
	\centering
	\begin{tabular}{ccccccc}
		\toprule
		m	  & $1$ & $1.5$ & $2$ & $10$ & $10^2$ & $10^3$ \\
		\midrule
		Order & 0.999 & 1.000 & 1.000 & 0.980 & 0.927 & 0.918 \\
		Est. order & $--$ & 0.332 & 0.499 & 0.899 & 0.990 & 0.999 \\
		\bottomrule
	\end{tabular}
	\label{tab:Ex2}
	\caption{Simulated order of convergence of the Midpoint rule (\ref{eqn:Midpoint}) applied to (\ref{eqn:Ex2}). Calculation has been done at $x=0.001$.}
\end{table}

\section{Conclusion}
We have shown that an explicit method for solving the Volterra integral equation with a power-type nonlinearity is convergent. The subjects of the future work consist of generalizing our results to the implicit schemes for integral equations with kernels bounded from below by an arbitrary power functions. 

\section*{Acknowledgements}
This research was supported by the National Science Centre, Poland under the project with a signature NCN $2015/17/D/ST1/00625$.\\

\noindent\L. P. would like to thank Prof. W. Okrasi{\'n}ski for a lot of fruitful talks that shaped the final form and quality of the paper.


\begin{thebibliography}{10}
	
	\bibitem{Ame97}
	William~F Ames and BG~Pachpatte.
	\newblock {\em Inequalities for differential and integral equations}, volume
	197.
	\newblock Academic press, 1997.
	
	\bibitem{Atk76}
	Kendall Atkinson.
	\newblock A survey of numerical methods for the solution of fredholm integral
	equations of the second kind.
	\newblock 1976.
	
	\bibitem{Bru84}
	Hermann Brunner.
	\newblock Iterated collocation methods and their discretizations for volterra
	integral equations.
	\newblock {\em SIAM journal on numerical analysis}, 21(6):1132--1145, 1984.
	
	\bibitem{Bru04}
	Hermann Brunner.
	\newblock {\em Collocation methods for Volterra integral and related functional
		differential equations}, volume~15.
	\newblock Cambridge University Press, 2004.
	
	\bibitem{Bru17}
	Hermann Brunner.
	\newblock {\em Volterra Integral Equations: An Introduction to Theory and
		Applications}, volume~30.
	\newblock Cambridge University Press, 2017.
	
	\bibitem{Buc97}
	Evelyn Buckwar.
	\newblock {\em Iterative Approximation of the Positive Solutions of a Class of
		Nonlinear Volterra-type Integral Equations}.
	\newblock Logos Verlag, 1997.
	
	\bibitem{Cor00}
	Evelyn Buckwar.
	\newblock On a nonlinear volterra integral equation.
	\newblock In {\em Volterra equations and applications}, pages 157--162. CRC
	Press, 2000.
	
	\bibitem{Buc05}
	Evelyn Buckwar.
	\newblock Existence and uniqueness of solutions of abel integral equations with
	power-law non-linearities.
	\newblock {\em Nonlinear Analysis: Theory, Methods \& Applications},
	63(1):88--96, 2005.
	
	\bibitem{Bus76}
	PJ~Bushell.
	\newblock On a class of volterra and fredholm non-linear integral equations.
	\newblock In {\em Mathematical Proceedings of the Cambridge Philosophical
		Society}, volume~79, pages 329--335. Cambridge Univ Press, 1976.
	
	\bibitem{Bus86}
	PJ~Bushell.
	\newblock The cayley-hilbert metric and positive operators.
	\newblock {\em Linear Algebra and its Applications}, 84:271--280, 1986.
	
	\bibitem{Bus89}
	PJ~Bushell and W~Okrasinski.
	\newblock Uniqueness of solutions for a class of non-linea volterra integral
	equations with convolution kernel.
	\newblock In {\em Mathematical Proceedings of the Cambridge Philosophical
		Society}, volume 106, pages 547--552. Cambridge University Press, 1989.
	
	\bibitem{Caf09}
	Luis Caffarelli and Luis Silvestre.
	\newblock Regularity theory for fully nonlinear integro-differential equations.
	\newblock {\em Communications on Pure and Applied Mathematics}, 62(5):597--638,
	2009.
	
	\bibitem{Cor91}
	Constantin Corduneanu.
	\newblock {\em Integral equations and applications}, volume 148.
	\newblock Cambridge University Press Cambridge, 1991.
	
	\bibitem{Ded05}
	Lj~Dedi{\'c}, M~Mati{\'c}, and J~Pe{\v{c}}ari{\'c}.
	\newblock On euler midpoint formulae.
	\newblock {\em The ANZIAM Journal}, 46(03):417--438, 2005.
	
	\bibitem{Dix86}
	J~Dixon and S~McKee.
	\newblock Weakly singular discrete gronwall inequalities.
	\newblock {\em ZAMM-Journal of Applied Mathematics and Mechanics/Zeitschrift
		f{\"u}r Angewandte Mathematik und Mechanik}, 66(11):535--544, 1986.
	
	\bibitem{Gor97}
	Rudolf Gorenflo and Francesco Mainardi.
	\newblock {\em Fractional calculus}.
	\newblock Springer, 1997.
	
	\bibitem{Gri81}
	Gustaf Gripenberg.
	\newblock Unique solutions of some volterra integral equations.
	\newblock {\em Mathematica Scandinavica}, 48(1):59--67, 1981.
	
	\bibitem{Gri90}
	Gustaf Gripenberg, Stig-Olof Londen, and Olof Staffans.
	\newblock {\em Volterra integral and functional equations}, volume~34.
	\newblock Cambridge University Press, 1990.
	
	\bibitem{Hac12}
	Wolfgang Hackbusch.
	\newblock {\em Integral equations: theory and numerical treatment}, volume 120.
	\newblock Birkh{\"a}user, 2012.
	
	\bibitem{Kel81}
	Jakob~J Keller.
	\newblock Propagation of simple non-linear waves in gas filled tubes with
	friction.
	\newblock {\em Zeitschrift f{\"u}r Angewandte Mathematik und Physik (ZAMP)},
	32(2):170--181, 1981.
	
	\bibitem{Kre89}
	Rainer Kress, V~Maz'ya, and V~Kozlov.
	\newblock {\em Linear integral equations}, volume~17.
	\newblock Springer, 1989.
	
	\bibitem{Lin85}
	Peter Linz.
	\newblock {\em Analytical and numerical methods for Volterra equations}.
	\newblock SIAM, 1985.
	
	\bibitem{Mai10}
	Francesco Mainardi.
	\newblock {\em Fractional calculus and waves in linear viscoelasticity: an
		introduction to mathematical models}.
	\newblock World Scientific, 2010.
	
	\bibitem{Man51}
	W~Robert Mann and Franti{\v{s}}ek Wolf.
	\newblock Heat transfer between solids and gases under nonlinear boundary
	conditions.
	\newblock {\em Quarterly of Applied Mathematics}, 9(2):163--184, 1951.
	
	\bibitem{Mc82}
	S~McKee.
	\newblock Generalised discrete gronwall lemmas.
	\newblock {\em ZAMM-Journal of Applied Mathematics and Mechanics/Zeitschrift
		f{\"u}r Angewandte Mathematik und Mechanik}, 62(9):429--434, 1982.
	
	\bibitem{Mil93}
	Kenneth~S Miller and Bertram Ross.
	\newblock An introduction to the fractional calculus and fractional
	differential equations.
	\newblock 1993.
	
	\bibitem{Myd91}
	W~Mydlarczyk.
	\newblock The existence of nontrivial solutions of volterra equations.
	\newblock {\em Mathematica Scandinavica}, pages 83--88, 1991.
	
	\bibitem{Nie97}
	Juan~J Nieto and W~Okrasinski.
	\newblock Existence, uniqueness, and approximation of solutions to some
	nonlinear diffusion problems.
	\newblock {\em Journal of Mathematical Analysis and Applications},
	210(1):231--240, 1997.
	
	\bibitem{Okr89}
	W~Okrasinski.
	\newblock Nonlinear volterra equations and physical applications appendix: on a
	extension of gripenberg's condition.
	\newblock {\em Extracta mathematicae}, 4(2):51--80, 1989.
	
	\bibitem{Okr95}
	W~Okrasinski.
	\newblock On nontrivial solutions to some nonlinear ordinary differential
	equations.
	\newblock {\em Journal of mathematical analysis and applications},
	190(2):578--583, 1995.
	
	\bibitem{O12}
	Donal O'Regan and Maria Meehan.
	\newblock {\em Existence theory for nonlinear integral and integrodifferential
		equations}, volume 445.
	\newblock Springer Science \& Business Media, 2012.
	
	\bibitem{Plo14}
	{\L}ukasz P{\l}ociniczak.
	\newblock Approximation of the erde?lyi--kober operator with application to the
	time-fractional porous medium equation.
	\newblock {\em SIAM Journal on Applied Mathematics}, 74(4):1219--1237, 2014.
	
	\bibitem{Plo15}
	{\L}ukasz P{\l}ociniczak.
	\newblock Analytical studies of a time-fractional porous medium equation.
	derivation, approximation and applications.
	\newblock {\em Communications in Nonlinear Science and Numerical Simulation},
	24(1):169--183, 2015.
	
	\bibitem{Plo13}
	{\L}ukasz P{\l}ociniczak and Hanna Okrasi{\'n}ska.
	\newblock Approximate self-similar solutions to a nonlinear diffusion equation
	with time-fractional derivative.
	\newblock {\em Physica D: Nonlinear Phenomena}, 261:85--91, 2013.
	
	\bibitem{Plo17}
	{\L}ukasz P{\l}ociniczak and Szymon Sobieszek.
	\newblock Numerical schemes for integro-differential equations with
	erd{\'e}lyi-kober fractional operator.
	\newblock {\em Numerical Algorithms}, pages 1--26.
	
	\bibitem{Sam93}
	Stefan~G Samko, Anatoly~A Kilbas, Oleg~I Marichev, et~al.
	\newblock Fractional integrals and derivatives.
	\newblock {\em Theory and Applications, Gordon and Breach, Yverdon}, 1993,
	1993.
	
	\bibitem{Sta00}
	Ivar Stakgold.
	\newblock {\em Boundary Value Problems of Mathematical Physics: Volume 1}.
	\newblock SIAM, 2000.
	
	\bibitem{Tao03}
	L{\"u}~Tao and Huang Yong.
	\newblock A generalization of discrete gronwall inequality and its application
	to weakly singular volterra integral equation of the second kind.
	\newblock {\em Journal of Mathematical Analysis and Applications},
	282(1):56--62, 2003.
	
	\bibitem{Tit86}
	Edward~Charles Titchmarsh and David~Rodney Heath-Brown.
	\newblock {\em The theory of the Riemann zeta-function}.
	\newblock Oxford University Press, 1986.
	
	\bibitem{Tri57}
	Francesco~Giacomo Tricomi.
	\newblock {\em Integral equations}, volume~5.
	\newblock Courier Corporation, 1957.
	
	\bibitem{Vaz07}
	Juan~Luis V{\'a}zquez.
	\newblock {\em The porous medium equation: mathematical theory}.
	\newblock Oxford University Press, 2007.
	
	\bibitem{Wan14}
	JinRong Wang, Chun Zhu, and Michal Fe{\v{c}}kan.
	\newblock Analysis of abel-type nonlinear integral equations with weakly
	singular kernels.
	\newblock {\em Boundary Value Problems}, 2014(1):20, 2014.
	
	\bibitem{Ye07}
	Haiping Ye, Jianming Gao, and Yongsheng Ding.
	\newblock A generalized gronwall inequality and its application to a fractional
	differential equation.
	\newblock {\em Journal of Mathematical Analysis and Applications},
	328(2):1075--1081, 2007.
	
\end{thebibliography}
\end{document}